\documentclass[12pt]{amsart}
\usepackage{amssymb,latexsym,amsmath,amsfonts}
\usepackage[mathscr]{eucal}
\usepackage{mathrsfs}
\usepackage{graphicx}
\usepackage{amscd}
\usepackage{amsthm}
\usepackage{amsxtra}
\usepackage[nointegrals]{wasysym}
\usepackage{bm}
\usepackage{float}
\usepackage[usenames,dvipsnames]{xcolor}
\usepackage{hyperref}
\hypersetup{colorlinks=true, citecolor=MidnightBlue, linkcolor=BrickRed}

\numberwithin{equation}{section}
\theoremstyle{definition}
\newtheorem{definition}{Definition}[section]
\theoremstyle{definition}
\theoremstyle{plain}
\newtheorem{theorem}[definition]{Theorem}
\newtheorem{cor}[definition]{Corollary}

\newcommand{\beas}{\begin{eqnarray*}}
\newcommand{\eeas}{\end{eqnarray*}}
\newcommand{\bes} {\begin{equation*}}
\newcommand{\ees} {\end{equation*}}
\newcommand{\be} {\begin{equation}}
\newcommand{\ee} {\end{equation}}
\newcommand{\bea} {\begin{eqnarray}}
\newcommand{\eea} {\end{eqnarray}}

\newcommand{\bdy}{\partial}
\newcommand{\Om}{\Omega}
\newcommand{\D}{\mathbb{D}}
\newcommand{\cont}{\mathcal{C}}
\newcommand\hR[1]{\widehat{#1}_{\operatorname{rat}}}
\newcommand\wt[1]{\widetilde{#1}}

\newcommand{\hol}{\mathcal{H}}
\newcommand{\Po}{\mathcal{P}}
\newcommand{\R}{\mathcal{R}}
\newcommand{\T}{\mathcal{T}}
\newcommand{\Z}{\mathcal{Z}}
\newcommand{\CC}{\mathbb{C}^2}
\newcommand{\Cn}{\mathbb{C}^n}
\newcommand{\C} {\mathbb{C}} 
\newcommand{\rl}{\mathbb{R}}

\newcommand\EatDot[1]{}

\begin{document}
\title[Rudin's Klein bottle]{The Rational hull of Rudin's Klein bottle}
\author{John T. Anderson}

\address{Department of Mathematics and Computer Science, College of the Holy Cross,
Worcester, Massachusetts 01610}
\email{anderson@mathcs.holycross.edu} 

\author{Purvi Gupta}

\address{Department of Mathematics, Rutgers University\\
New Brunswick, New Jersey 08854}
\email{purvi.gupta@rutgers.edu }

\author{Edgar L. Stout}

\address{Department of Mathematics, University of Washington,
Seattle, Washington 98195}
\email{stout@math.washington.edu}

\begin{abstract}
In this note, a general result for determining the rational hulls of fibered sets in $\CC$ is established. We use this to compute the rational hull of Rudin's Klein bottle, the first explicit example of a totally real nonorientable surface in $\CC$. In contrast to its polynomial hull, which was shown to contain an open set by the first author in 2012, its rational hull is shown to be two-dimensional. Using the same method, we also compute the rational hulls of some other surfaces in $\CC$. 	
\end{abstract}
%\keywords{}
%\subjclass{}
\maketitle
\section{Introduction}\label{sec_intro}

In \cite{Ru81}, Rudin gave an explicit smooth embedding of the Klein bottle into $\CC$ so that the image is {\em totally real}, i.e., no tangent space to the image in $\CC$ contains a nontrivial complex subspace. Later, Givental \cite{Gi86} constructed totally real embeddings into $\CC$ of all nonorientable surfaces that are connected sums of $n$ Klein bottles, where $n$ is odd and at least three. In fact, Givental's embeddings are {\em Lagrangian}, i.e., the pull-backs of the standard K{\"a}hler form via the embeddings vanish. The question of whether or not the Klein bottle admits {\em any} Lagrangian embedding into $\CC$ was settled much later (in the negative) by Shevchishin (see \cite{Sh09}). This places certain constraints on the convexity properties of totally real Klein bottles in $\CC$, as discussed below.   

Given a compact set $X\subset\Cn$, its {\em rational hull} is defined as 
	\bes
		\hR X=\{z\in\Cn: P(z)\in P(X)\ \text{for all polynomials}\ P:\Cn\rightarrow\C \}.
	\ees
The set $X$ is said to be {\em rationally convex} if its rational hull is trivial, i.e., $\hR X=X$. One can alternatively identify $\hR X$ with the maximal ideal space of the {\em rational algebra of $X$}, which is defined as  
	\bes
		\R(X)=\left\{f\in\cont(X):
			\!\begin{array}{c}
			f\ \text{is the uniform limit on $X$ of}\\
			 \text{rational functions with poles off $X$}
			\end{array}\!\right\}. 
	\ees
We note that if $\cont(X)=\R(X)$, then $X$ is rationally convex. If, further, $X\subset\Cn$ is a smooth totally real submanifold, then the converse is also true (see \cite[Section 6.3]{St07}). In \cite{DuSi95}, Duval and Sibony show that for any $n$-dimensional totally real submanifold of $\Cn$, being rationally convex is equivalent to being Lagrangian with respect to some K{\"a}hler form on $\Cn$. So, although there are rationally convex topological Klein bottles in $\CC$ (see \cite{ShSu15}), by Shevchishin's result, every totally real Klein bottle in $\CC$ must have a nontrivial rational hull. It is natural to ask whether there is a constraint on the dimension of this hull. In this paper, we show that the rational hull of Rudin's Klein bottle is $2$-dimensional. More precisely, the hull consists of the Klein bottle and an attached analytic annulus (see Section~\ref{sec_klein1}). We also characterize the rational algebra of Rudin's Klein bottle.

A similar question regarding the polynomial hull of Rudin's Klein bottle is addressed in \cite{An12}. The {\em polynomial hull} of a compact set $X\subset\Cn$ is
	\bes
		\{z\in\Cn:|P(z)|\leq \sup_X|P|\ \text{for all polynomials}\ P:\Cn\rightarrow\C\}.
	\ees
It is known \cite{Al72} that the polynomial hull of any compact $n$-dimensional manifold in $\Cn$ must be of dimension at least $n+1$. In \cite{An12}, the first author shows that the polynomial hull of Rudin's Klein bottle contains an open set in $\CC$. On the other hand, he produces a totally real Klein bottle in $\CC$ whose polynomial hull is of dimension $3$. In spite of admitting a smaller polynomial hull, the rational hull of his Klein bottle is qualitatively the same as that of Rudin's Klein bottle since they are equivalent under a rational automorphism. (See Section~\ref{sec_klein1} for details.) 

To compute the rational hull of Rudin's Klein bottle, we will establish a general result in the fashion of the following well-known criterion for fibered sets in $\Cn$ (see \cite[Theorem 1.2.16]{St07}): {\em If $X\subset\Cn$ is compact and if $f\in\Po(X)$ satisfies $\Po(f(X))=\cont(f(X))$, then $X$ is polynomially convex if and only if each fiber $f^{-1}(t)$, $t\in f(X)$, is polynomially
convex. If $X$ is polynomially convex, then $\Po(X)=\cont(X)$ if and only if for each $t\in f(X)$, $\Po(f^{-1}(t))=\cont(f^{-1}(t))$.} Here $\Po(X)$ denotes the algebra of functions on $X$ that are uniformly approximable on $X$ by polynomials.

The analogous statement for rational convexity is not true in general; the surfaces considered in this paper are counterexamples. The key is to impose a stronger condition than that of rational convexity of fibers --- namely, the rational convexity of fibers with respect to a common analytic curve. Given an analytic curve $V\subset\CC$, a compact set $X\subset\CC\setminus V$, and an entire function $F$ so that $V=\{F=0\}$, we denote by $\R_V(X)$ the uniform closure in $\cont(X)$ of the set 
	\be\label{eq_alg}
		\left\{\left.\frac{G}{F^m}\right|X: G\ \text{is entire}, m\in\mathbb N\right\}. 
	\ee
Note that $\R_V(X)\subset\R(X)$. We now state the general result.
%-----------------
%-----------------
%-----------------
%-----------------
%-----------------
\begin{theorem}\label{thm_main} Let $X\subset\CC$ be a compact set and $f:\CC\rightarrow\C$ be a rational function with no poles on $X$. Suppose $\Gamma=f(X)$ satisfies $\R(\Gamma)=\cont(\Gamma)$.  For $t\in\Gamma$, let $X_t=f^{-1}(t)\cap X$. Let $V$ be an analytic curve in $\CC$ that avoids $X$.
 Then, 
	\bes
	\hR X\subseteq X\cup\bigcup_{t\in\Gamma}\Om_t
	\ees
where $\Om_t$ denotes the union of all the bounded components of $f^{-1}(t)\setminus X_t$ that avoid $V$. Furthermore,  
	\bes
		\R(X)\supseteq\{\psi\in\cont(X):\psi|{X_t}\in \R_V(X_t)\ \text{for all}\ t\in\Gamma\}.
	\ees
In particular, if $\R_V(X_t)=\cont(X_t)$ for all $t\in \Gamma$, then $\R(X)=\cont(X)$ (and $X$ is rationally convex). 
\end{theorem}

We must clarify that if $f:\CC\rightarrow\C$ is a rational function of the form $p/q$, where $p$ and $q$ are relatively prime polynomials, then by $f^{-1}(t)$, $t\in\C$, we mean $\{z\in\CC:p(z)=tq(z)\}$.

The above result can also be applied to other examples present in the literature. A few of these have been collected in Section~\ref{sec_other}. In each of these cases, the rational hull comes from an attached annulus. Note that since the above proof only gives a partial description of the rational hull, we need the following well-known principle to complete our computations: {\em Suppose $X\subset \Cn$ is a compact set and $\Sigma\subset\Cn$ is a bordered Riemann surface whose smooth boundary $\bdy\Sigma$ is in $X$, and bounds an oriented surface $S\subset X$ in the following sense: for any smooth $1$-form $\alpha$ defined in a neighborhood of $X\cup\Sigma$, we have
	\bes
		\int_{S}d\alpha=\int_{\bdy\Sigma}\alpha.
	\ees
Then, $\Sigma\subset\hR X$.} In the case of Rudin's Klein bottle, a modified version of this principle applies, where we must allow $S$ to have multiplicity at $\bdy\Sigma$. We provide an alternate argument in Section \ref{sec_klein1}. 

Lastly, our computations show that Rudin's Klein bottle is an example of a surface $S$ with the property that $\hR S\setminus S$ is a smooth $2$-manifold. Examples of $2$-dimensional tori with this property were earlier constructed by Duval and Sibony in \cite{DuSi98} (see Section~\ref{sec_other}). These examples all address a question raised by Alexander Izzo (private communication) of whether such surfaces exist in $\CC$.  
	
\noindent{\bf Acknowledgements.} We would like to extend our gratitude to the anonymous referee whose comments led to the general result (Theorem~\ref{thm_main}) in this article, to Harold Boas whose careful reading of the manuscript vastly improved its presentation, and to Alexander Izzo who played a key role in facilitating this collaboration.

%------------
%------------
%---Section 2----
%------------
%------------
\section{Proof of Theorem~\ref{thm_main}}\label{sec_proof}

We will use the observation that if $Y\subset\Cn$ is compact, then $y\notin\hR Y$ if and only if there is an entire function $g$ on $\Cn$ such that $g(y)\notin g(Y)$. 

\begin{proof} Suppose $f=p/q$ for relatively prime polynomials $p$ and $q$ on $\CC$. Let $(z_0,w_0)\in\hR X\setminus X$. For a compact set $Y$ and a rational
function $g$ with no poles on $Y$, $g(y)\in g(Y)$ for each point $y\in\hR Y$. Thus, $t_0=f(z_0,w_0)\in\Gamma$.
%We write $f=\frac{p}{q}$, where $p$ and $q$ are relativly prime polynomials on $\CC$. Then, $q(z_0,w_0)\neq 0$, and $t_0=f(z_0,w_0)\in \Gamma$. Else, either $q$ or $p-t_0q$ separates $(z_0,w_0)$ from $X$. 
Now, choose a $g\in\cont(\Gamma)$ such that $g(t_0)=1$ and $|g|<1$ on $\Gamma\setminus\{t_0\}$. Then, since $\R(\Gamma)=\cont(\Gamma)$, $(g^n\circ f)h\in\R(X)$ for all $n\in\mathbb N$ and all $h\in\R(X)$. If $\mu$ is a positive representing measure for the point $(z_0,w_0)$ with respect to the algebra $\R(X)$, then 
	\bes
		|h(z_0,w_0)|=\left|\int_Xh(z,w)(g^n\circ f)(z,w)d\mu(z,w)\right|.
	\ees
Letting $n\rightarrow\infty$ yields
	\be\label{eq_pol}
		|h(z_0,w_0)|=\left|\int_{X_{t_0}}h(z,w)d\mu(z,w)\right|\leq \sup_{X_{t_0}}|h|.
	\ee
In particular, \eqref{eq_pol} holds for all polynomials $h:\CC\rightarrow\C$. Thus, $(z_0,w_0)$ is in the polynomial hull of $X_{t_0}$. Since $f^{-1}(t_0)$ is a variety in $\CC$, $(z_0,w_0)$ must be contained in a bounded component $D$ of $f^{-1}(t_0)\setminus X_{t_0}$. 

Now, suppose $V$ meets $D$, where we write $V=\{F=0\}$ for some entire function $F$. Let $G:\CC\setminus V\rightarrow \C^3_{z,w,\zeta}$ be the map 
	\bes
		(z,w)\mapsto \left(z,w,\frac{1}{F(z,w)}\right).
	\ees
Then, $G(D\setminus V)$ is an unbounded component of $V'\setminus G(X_{t_0})$, where 
	\bes
		V'=\{(z,w,\zeta)\in\C^3:p(z,w)=t_0q(z,w),\zeta F(z,w)=1\}.
	\ees
Since $V'$ is an analytic curve, $G(z_0,w_0)\in G(D\setminus V)$ is not in the polynomial hull of $G(X_{t_0})$. So, there is a polynomial $P:\C^3\rightarrow\C$ with $|P\circ G(z_0,w_0)|>\sup_{X_{t_0}}|P\circ G|$. Since $P\circ G\in \R(X)$, this contradicts \eqref{eq_pol}. Thus, $D\subseteq\Om_{t_0}$. 

We now prove the second part of the theorem. We abuse notation to denote $f^{-1}(A)\cap X$ simply by $f^{-1}(A)$, for any set $A\subset\Gamma$. Let $\mu$ be an extreme point of the closed unit ball of $\R(X)^\perp$, the space of finite regular Borel measures orthogonal to $\R(X)$. Then, since $\R(\Gamma)=\cont(\Gamma)$, $\mu$ is supported on a single fiber $X_t$. If not, then there is a decomposition $\Gamma=E\cup E'$ with $E$ and $E'$ measurable subsets of $\Gamma$ with the property that $\mu$ has positive mass on both $f^{-1}(E)=\cup_{t\in E}X_t$ and $f^{-1}(E')=\cup_{t\in E'}X_t$. If $K$ is a compact subset of $E$, and $g$ is a continuous function on $\Gamma$ with $g=1$ on $K$ and $|g|<1$ off $K$, then for each $h\in\R(X)$, we have
	\bes
		\int_X(g^n\circ f)(x)h(x)d\mu(x)=0.
	\ees
As $n\rightarrow\infty$, the integral on the left-hand side converges to $\int_{f^{-1}(K)}h(x)d\mu(x)$. Thus, $\mu|f^{-1}(K)$ is orthogonal to $\R(X)$. This is true for every choice of $K$, so the measure $\mu_E=\mu|f^{-1}(E)\in\R(X)^\perp$. By a similar argument, $\mu_{E'}=\mu|f^{-1}(E')\in\R(X)^\perp$. Since $\mu=\mu_E+\mu_{E'}$, we get a contradiction to the fact that $\mu$ is an extreme point of the closed unit ball of $\R(X)^\perp$. Thus, $\mu$ must be supported on some fiber $X_t$. 

Now, let $\psi\in\cont(X)$ be such that $\psi|{X_t}\in \R_V(X_t)$ for all $t\in\Gamma$. Since, $\psi|{X_t}\in \R_V(X_t)$, there is a sequence $\{g_n\}_{n\in\mathbb N}$ of functions of the form $G/F^m$, $G$ entire and $m\in\mathbb N$ (see \eqref{eq_alg}), so that $\{g_n|{X_t}\}_{n\in\mathbb N}$ converges to $\psi|{X_t}$. Since $\R(\Gamma)=\cont(\Gamma)$, $X_t$ is a peak set for $\R(X)$ and $\R(X)|X_t$ --- the algebra of restrictions to $X_t$ of functions in $\R(X)$ --- is closed. Since each $g_n\in\R(X)$, $\psi|{X_t}\in\R(X)|X_t$. So, there is an $\wt \psi\in \R(X)$ that restricts to $\psi$ on $X_t$. Thus, with $\mu$ an extreme point of the closed unit ball of $\R(X)^\perp$, we have that
	\bes
		\int_X\psi\:d\mu=\int_{X_t}\psi\:d\mu=\int_{X_t}(\psi-\wt \psi+\wt \psi)\:d\mu
			=\int_X\wt \psi\:d\mu=0. 
	\ees
Thus, $\psi\in\R(X)$. If we further have that $\R_V(X_t)=\cont(X_t)$ for all $t\in \Gamma$, then 
	\bes
		\cont(X)
			\supseteq\R(X)\supseteq\{\psi\in\cont(X):\psi|{X_t}\in \cont(X_t)\ \text{for all}\ t\in\Gamma\}
			=\cont(X).
		\ees
Hence, the claim. 
\end{proof}

%------------
%------------
%---Section 3----
%------------
%------------
\section{Rudin's Klein Bottle}\label{sec_klein1}
In this section, we apply Theorem~\ref{thm_main} to compute the rational hull of 
	\be\label{eq_gh}
		K=\left\{(e^{2 i\theta}g^2(\phi),
						e^{i\theta}g(\phi)h(\phi))\in\CC
							:-\pi\leq \theta,\phi<\pi  \right\},
	\ee
where $g(\phi)=a+b\cos\phi$ and $h(\phi)=\sin\phi+i\sin 2\phi$, for some fixed real numbers $0<b<a$.

%-----------Theorem 1-------------------
\begin{theorem}\label{thm_klein1}
The rational hull of $K$ is $K\cup A$, where 
	\bes
		A=\{(z,0)\in\CC: (a-b)^2\leq |z|\leq (a+b)^2\}.
	\ees
Furthermore, if $A^\circ$ denotes the interior of the annulus $A$ relative to  the plane $w=0$, then
\bes
	\R(K)=\{\psi\in\cont(K): \psi\ \text{extends holomorphically to}\ A^\circ\}.
\ees
\end{theorem}

%-----------of Proof Theorem 1-------------------
\begin{proof}
We first show that $A\subset \hR K$. Let $\D$ denote the open unit disc in $\C$. Consider the continuous family of maps $F_\phi:\overline\D\rightarrow\CC$ given by $F_\phi:\zeta\mapsto\big(\zeta^2g^2(\phi),
						\zeta g(\phi)h(\phi)\big)$,
where $\phi\in[-\pi,0]$. Then, $\{F_\phi(\cdot)\}_{[-\pi,0]}$ is a continuous family of holomorphic discs in $\CC$ whose boundaries are attached to $K$. Note that $\zeta\mapsto F_\phi(\zeta)$ is a two-to-one map on $\overline\D\setminus{0}$, when $\phi$ is either $-\pi$ or $0$. Moreover, 
	\beas
		F_{-\pi}(\overline\D)&=&\{(z,0)\in\CC:|z|\leq (a-b)^2\}\\
		F_0(\overline\D)&=&\{(z,0)\in\CC:|z|\leq (a+b)^2\}.
	\eeas
Thus, for any holomorphic function $f$ defined on $\CC$, 
	\be\label{eq_count}
		\#\Z(f\circ F_0)-\#\Z(f\circ F_{-\pi})=2\: \#\Z(f|A),
	\ee
for $f$ nonvanishing on $\bdy A$, where $\#\Z(g)$ denotes the number of zeros of $g$ (counting multiplicities). In particular, suppose $P$ is a polynomial that does not vanish anywhere on $K$. Then, each $f_\phi=P\circ F_\phi$ is holomorphic on $\D$, continuous up to the boundary, and nonvanishing on $\bdy\D$. So,  
		\bes
		G:\phi\mapsto\#\Z(f_\phi) =\frac{1}{2\pi i}\int_{\bdy\D}
				\frac{f_\phi '(\zeta)}{f_\phi(\zeta)}d\zeta
	\ees
is a continuous integer-valued function on $[-\pi,0]$ --- therefore, it must be a constant. From \eqref{eq_count}, we have have that
	\bes
		\#\Z(P|A)=\frac{\#\Z(f_0)-\#\Z(f_{-\pi})}{2}
		=\frac{G(0)-G(-\pi)}{2}= 0.
	\ees
Thus, $P$ cannot vanish on $A$. Since $P$ was an arbitrarily chosen polynomial that doesn't vanish on $K$, $A\subset\hR K$.

%----------------------Part II------------------------------------------
For the rest of the proof, we apply Theorem~\ref{thm_main} with 
	\bes
		f(z,w)=\dfrac{w^2}{z},\ V=\{(z,w)\in\CC:z=0\},\ \text{and}\ F(z,w)=z.
	\ees
In this case, the set 
	\bes
		\Gamma=f(K)=\{(\sin\phi+i\sin 2\phi)^2:-\pi\leq \phi<0\}
	\ees
 is a simple closed curve in $\C$, and satisfies $\R(\Gamma)=\cont(\Gamma)$. For $t\in\Gamma$, we let $K_t=f^{-1}(t)\cap K$. Then,
	\bea\label{eq_fiber}
		K_t
			=\begin{cases}
				C_{-\pi}\cup C_0, &\text{when}\ t=0,	\\			
				C_{\phi},\ \text{for some}\
					 \phi\in(0,\pi),\ &\text{when}\  t\neq 0.
			\end{cases}	
	\eea
where 
$C_\phi=\{(e^{2i\theta}g^2(\phi), e^{i\theta}g(\phi)h(\phi)):-\pi\leq \theta<\pi\}$ 
is a circle on $K$.  Indeed, if $t=h^2(\phi)=0$, then $\phi$ is either $0$ or $-\pi$. On the other hand, if $t\neq 0$, then $h^2(\phi)=t$ yields exactly two solutions in the interval $(-\pi,\pi)$, since 
$h(-\phi)=-h(\phi)$. So, $K_t=C_\phi\cup C_{-\phi}$ for some $\phi\in(0,\pi)$. However,
	\beas
		C_{\phi}
		&=&\{(e^{2i\theta}g^2(\phi), e^{i\theta}g(\phi)h(\phi)):\theta\in\rl\}\\
		&=&\{(e^{2i\eta}g^2(-\phi), e^{i\eta}g(-\phi)h(-\phi)):\eta\in\rl\}
		= C_{-\phi}. 
	\eeas 

Now, if $t\neq 0$, then the only bounded component of $f^{-1}(t)\setminus K_t$ is a holomorphic disc intersecting $V$. For $t=0$, the set $f^{-1}(0)\setminus K_0$ consists of two bounded components --- the flat holomorphic disc $\{(z,0),|z|<(a-b)^2\}$, which intersects $V$, and the annulus $A$, which avoids $V$. Thus, by Theorem~\ref{thm_main}, $\hR K\subset K\cup A$.

For the characterization of $\R(K)$, first observe that since it can be identified with $\R(\hR K)$, the inclusion $\R(K)\subseteq \{\psi\in\cont(K): \psi|\bdy A\ \text{extends holomorphically to}\ A^\circ\}$ follows. To get the other inclusion, we need to compute $\R_V(K_t)$ for all $t\in\Gamma$. We claim that
	\bes
		\R_V(K_t)=
					\begin{cases}
				\hol(A), &\text{when}\ t=0,	\\			
					 \cont(K_t),\ &\text{when}\  t\neq 0,
			\end{cases}	
	\ees
where $\hol(A)=\{\psi\in\cont(\bdy A):\psi\ \text{extends holomorphically to}\ A^\circ\}$. Indeed, for a fixed $\phi$, we have that $z^j|K\in\R_V(C_\phi)$ for all positive and negative integers $j$, and $w^\ell|K\in\R_V(C_\phi)$ for all positive integers $\ell$. It follows that $\R_V(C_\phi)$ contains $e^{ik\theta}$ for all integers $k$. So, if $t\neq 0$, then $\R_V(K_t)=\R(K_t)=\cont(K_t)$. On the other hand, if $t=0$, $\R_V(K_t)$ contains (the restrictions to $K_0$ of) all the Laurent polynomials in $z$. 
Thus, $\R_V(K_0)=\hol(A)$. This completes the characterization of $\R(K)$. 
\end{proof}

%---Corollary---
We now invoke \cite{Ro62} (and the references cited therein)  to obtain the following criterion as a corollary. It is immediate from the corollary that the (finite, regular, Borel) measures orthogonal to $\R(K)$ are those measures on $K$ that are concentrated on $\bdy A$ and that are orthogonal to the algebra $\R(A)$. 
 
\begin{cor}\label{cor_meas} The function $\psi\in\cont(K)$ lies in $\R(K)$ if and only if
$$\int_{\bdy A}\psi(z,0)g(z,0)\, dz=0$$
for all functions $g$ holomorphic on a neighborhood in the $z$--plane of the annulus $A$.
\end{cor}

In \cite{An12}, the first author modifies Rudin's Klein bottle to obtain the following totally real Klein bottle in $\CC$ whose polynomial hull is shown to be of dimension three. 
Let 
	\bes
		K^*=\left\{\left(e^{2 i\theta}g^2(\phi),
						e^{-i\theta}\frac{h(\phi)}{g(\phi)}\right)
							:-\pi\leq \theta,\phi<\pi  \right\},
	\ees
where $g$ and $h$ are as in \eqref{eq_gh}. We claim that the rational hull of $K^*$ is $K^*\cup A$, where $A$ is the annulus defined in Theorem~\ref{thm_main}. Indeed, consider the automorphism 
	\bes 
		J:(z,w)\mapsto\left(z,\frac{w}{z}\right)
	\ees
on the domain $\T=\{(z,w)\in\CC:z\neq 0\}$. The inverse of $J$  is given by $J^{-1}(z,w)=(z,zw)$. It is easy to see that the automorphism $J$ carries $K$ onto $K^*$, $J^{-1}(K^*)=K$, and it acts as the identity  map on the annulus $A$. Furthermore, $J$ effects an isomorphism between $\R(K^*)$ and $\R(K)$: given a rational function $f$ on $\CC$, $f$ is holomorphic on a neighborhood of $K^*$ if and only if $f\circ J$ is holomorphic on a neighborhood of $K$. It follows that $J$ restricts to a bijection between $\hR K$ and $\hR {K^*}$. The same reasoning gives the analogues of Theorem~\ref{thm_klein1} and Corollary~\ref{cor_meas} for $R(K^*)$.

%------------
%------------
%---Section 4----
%------------
%------------
\section{Other examples}\label{sec_other} 
We apply the general technique of this paper to compute the rational hulls of some other surfaces in $\CC$ that have appeared in the literature. The tori considered below are particular instances of a general scheme provided in \cite{DuSi98} to construct a totally real torus in $\CC$ whose rational hull consists precisely of the torus and $n$ attached annuli, where $n$ is any prescribed integer.

\subsection*{Totally real discs} In \cite{HoWe69}, and later in \cite{DuSi95}, totally real discs of the following form are considered: $D=\{(z,w)\in\CC:|z|\leq 1,w=\bar zf(|z|^2)\}$, with $f:[0,1]\rightarrow\C$ a $\cont^1$-smooth function such that $t\mapsto tf(t)$ is an immersion with a double point at some $t_0\in[0,1]$. Let $\alpha=t_0f(t_0)\in\C$. The curve $C=\{zw=\alpha\}$ intersects $D$ in two circles that bound a surface in $D$, and an annulus $A$ in $C$. So, $A\subset\hR D$. Setting $f(z,w)=zw$ and $V$ as any analytic curve that avoids $D$, we can apply Theorem~\ref{thm_main} to obtain that $\hR D=D\cup A$. 
  
\subsection*{Conjugate Hopf Tori} In \cite{DuGa14}, Duval and Gayet provide a dichotomy for any generic totally real unknotted torus embedded in $S^3$ --- either it is rationally convex and fillable by holomorphic discs, or its rational hull contains a holomorphic annulus. As an example of the latter, they consider a family of tori that come naturally from the conjugate Hopf fibration  $\Theta:S^3\rightarrow S^2\subset\C\times\rl$ given by $(z,w)\mapsto (2zw, |z|^2-|w|^2)$. Let $\pi$ denote the projection of $\C\times\rl$ onto $\C$. For any embedded closed curve $\gamma$ in $S^2$ such that $\pi(\gamma)\subset\overline\D$ is immersed, the set 
	\bes
		T^\gamma=\Theta^{-1}(\gamma)
	\ees
is a totally real torus in $S^3$. In the Duval-Gayet examples, $\gamma$ is such that $\pi(\gamma)\subset \overline\D$ is a figure eight avoiding the origin. 

Fix such a $\gamma$ and let $a$ be its point of self-intersection in $\D$. Let $f(z,w)=2zw$, $V=\{z=0\}$ and $F(z,w)=z$.  The fibers $f^{-1}(t)\cap{T^\gamma}$ are polynomially convex circles when $t\neq a$, and the fiber $f^{-1}(a)\cap T^\gamma$ bounds the annulus 
	\bes
		A
			=\left\{(z,w)\in \CC
				:2zw=a;
				 -\sqrt{1-a^2}\leq |z|^2-|w|^2\leq \sqrt{1-a^2} \right\}.
	\ees
Applying Theorem~\ref{thm_main}, and observing that $\bdy A$ bounds a cylinder in $T^\gamma$, we have that 
	\bes
		\hR {T^\gamma}=T^\gamma\cup A
	\ees
and
	\bes
	\R(T^\gamma)
		=\{\psi\in\cont(T^\gamma): \psi\ \text{extends holomorphically to}\ A^\circ\}
	\ees
where $A^\circ$ denotes the interior of the annulus $A$ relative to the variety $\{2zw=a\}$.

\subsection*{Spin Tori} Let $\gamma:[0,1]\rightarrow\C\times\rl$ be a smooth embedded curve, so that if $\gamma(\theta)=(z(e^{i\theta}),r(e^{i\theta}))$, then $\Gamma=z(e^{i\theta})$ is an immersed curve in $\C$ with one point of self-intersection, say $z(e^{i\theta_1})=z(e^{i\theta_2})=z_0$. Consider $T_\gamma:=\{(z(e^{i\theta}),r(e^{i\theta})e^{i\phi}):0\leq \theta,\phi\leq 2\pi\}$. Then, $T_\gamma$ is a totally real torus in $\CC$ and
	\bes
		\hR {T_\gamma}=T_\gamma\cup A
	\ees
and
	\bes
	\R(T_\gamma)
		=\{\psi\in\cont(T_\gamma): \psi\ \text{extends holomorphically to}\ A^\circ\}
	\ees
where $A=\{(z_0,w):\min(r(e^{i\theta_1}),r(e^{i\theta_2}))\leq |w|\leq \max(r(e^{i\theta_1}),r(e^{i\theta_2}))\}$.
%%%%%%%%%%%%%%%%%%%%%%%%%%%%%%%%%%%%%%%%%%%%%%%%%%

\bibliography{RHKleinBib}
\bibliographystyle{plain}
\end{document}